\documentclass[11pt]{article}
\usepackage{amsmath,amsfonts, amssymb, color,hyperref}

\numberwithin{equation}{section}

\newcommand{\be}{\begin{equation}}
\newcommand{\ee}{\end{equation}}
\newcommand{\bea}{\begin{eqnarray}}
\newcommand{\eea}{\end{eqnarray}}
\newcommand{\beas}{\begin{eqnarray*}}
\newcommand{\eeas}{\end{eqnarray*}}

\newtheorem{theorem}{Theorem}[section]

\newtheorem{proposition}[theorem]{Proposition}

\newtheorem{lemma}[theorem]{Lemma}

\newtheorem{example}[theorem]{Example}

\newtheorem{assumption}[theorem]{Assumption}

\newenvironment{Assumption}{\begin{assumption}\rm}{\end{assumption}}

\newenvironment{proof}{\addvspace{\medskipamount}\par\noindent{\it Proof}.}
{\unskip\nobreak\hfill$\Box$\par\addvspace{\medskipamount}}

\newcommand{\brak}[1]{\left(#1\right)}    
\newcommand{\crl}[1]{\left\{#1\right\}}   
\newcommand{\edg}[1]{\left[#1\right]}     

\newcommand{\qed}{\unskip\nobreak\hfill$\Box$\par\addvspace{\medskipamount}}

\newcommand{\p}{\mathbb{P}}

\def\1{{\mathbf 1}}

\title{BSDE formulation of combined regular\\ and singular stochastic control problems}

\author{Bruno Bouchard\thanks{Universit\'e Paris-Dauphine,  PSL Research University,  CNRS, CEREMADE, Paris, France (bouchard@ceremade.dauphine.fr). Partially supported by CAESARS - ANR-15-CE05-0024.} \and Patrick Cheridito\thanks{Department of Mathematics, ETH Zurich, R\"amistrasse 101, 8092 Zurich,
Switzerland.}
\and Ying Hu\thanks{Univ Rennes, CNRS, IRMAR - UMR 6625, F-35000 Rennes, France (ying.hu@univ-rennes1.fr) and School of Mathematical Sciences, Fudan University, Shanghai 200433, China.
Partially supported by Lebesgue center of mathematics ``Investissements d'avenir"
program - ANR-11-LABX-0020-01,  by CAESARS - ANR-15-CE05-0024 and by MFG - ANR-16-CE40-0015-01.}}

\date{January 2018}

\begin{document}

\maketitle

\begin{abstract} In this paper we study a class of combined  regular and
singular stochastic control problems that can be expressed as 
constrained BSDEs. In the Markovian case, this reduces to a characterization 
through a PDE with gradient constraint. But the BSDE formulation makes it possible to move beyond 
Markovian models and consider path-dependent problems. We also provide an approximation 
of the original control problem with standard BSDEs that yield a characterization of approximately 
optimal values and controls. 

\end{abstract}



\section{Introduction} 

We consider a class of continuous-time stochastic control problems involving two different controls:
a regular control affecting the state variable in an absolutely continuous way, and a singular 
control resulting in a cumulative impact of finite variation. For standard stochastic control in 
continuous time, we refer to the textbooks \cite{FR, K80, FS, YZ, Touzi}.
Singular stochastic control goes back to \cite{BC67a, BC67b} and has subsequently been studied by e.g.
\cite{BSW80, Kar81, Kar83, KS84, Kar85, KS85, CMR85, B87, MT87, EK88, K93, HS95}.
For a typical Markovian singular stochastic control problems, it can be deduced from 
dynamic programming arguments that the optimal value is given by a viscosity solution of
a PDE with gradient constraint. On the other hand, it has been shown 
that PDEs with gradient constraints are related to BSDEs with $Z$-constraints; 
see, e.g. \cite{BH2, PX13}. 

In this paper we directly show that a wide variety of combined regular and singular stochastic control
problems can be represented as $Z$-constrained BSDEs\footnote{Independently, Elie, Moreau and Possama\"{i}
have been working on a similar idea in \cite{EMP17}. But the exact class of control problems 
studied in \cite{EMP17} is different. Moreover, they employ analytic methods, while we use purely 
probabilistic arguments.}. This has the advantage that it allows to study path-dependent problems.
More precisely, we consider optimization problems of the form
\be \label{problem}
\sup_{\alpha, \beta} 
\mathbb{E} \edg{\int_0^T f(t,X^{\alpha,\beta},\alpha_t) dt + \int_0^T g(t,X^{\alpha,\beta},\alpha_t) d\beta_t 
+ h(X^{\alpha,\beta})}
\ee
for a $d$-dimensional controlled process with dynamics 
\be \label{contrdyn}
dX^{\alpha,\beta}_t = \mu(t,X^{\alpha,\beta}, \alpha_t)dt + \nu(t,X^{\alpha,\beta}, \alpha_t) d\beta_t 
+ \sigma(t,X^{\alpha,\beta}) dW_t, \quad X_0 = x \in \mathbb{R}^d,
\ee
where $(W_t)$ is an $n$-dimensional Brownian motion, $(\alpha_t)$ is a predictable process taking
values in a compact subset $A \subseteq \mathbb{R}^k$ (the regular control) and $(\beta_t)$ is an 
$l$-dimensional process with nondecreasing components (the singular control). 
The coefficients $\mu, \nu, \sigma$ and the functions $f,g,h$ are all allowed to depend in a 
non-anticipative way on the paths of $X^{\alpha,\beta}$.

Our main representation result is that the optimal value of \eqref{problem} is given by the 
initial value of the minimal supersolution of a BSDE 
\be \label{sBSDE}
Y_t = h(X) + \int_t^T p(s,X,Z_s) ds - \int_t^T Z_s dW_s 
\ee
subject to a constraint of the form $q(t,X,Z_t) \in \mathbb{R}^l_-$, where $(X_t)$ is the unique strong solution of 
an SDE
$$
dX_t = \eta(t,X) dt + \sigma(t,X) dW_t, \quad X_0 = x,
$$
with the same $\sigma$-coefficient as \eqref{contrdyn}.

In addition, we show that the original problem \eqref{problem} can be approximated with a sequence of 
standard BSDEs 
\be \label{jBSDE}
Y^j_t = h(X) + \int_t^T p^j(s,X,Z^j_s) ds - \int_t^T Z^j_s dW_s.
\ee
While the minimal supersolution of the constrained BSDE \eqref{sBSDE} gives the optimal value 
of the control problem \eqref{problem}, the BSDEs \eqref{jBSDE} can be used to characterize 
nearly optimal values as well as approximately optimal controls.

Due to the constraint $q$, it might happen that the minimal supersolution of \eqref{sBSDE} 
jumps at the final time $T$. In our last result, we show how this jump can be removed 
by replacing $h$ with the smallest majorant $\hat{h}$ of $h$ that is consistent with $q$ -- the 
so-called face-lift of $h$.

The rest of the paper is structured as follows. In Section \ref{sec:results}, we introduce the notation and 
our main results. All proofs are given in Section \ref{sec:proofs}. 

\section{Results}
\label{sec:results}

We consider a combined regular and singular stochastic control problem of the form 
\begin{equation} \label{I}
I := \sup_{(\alpha, \beta) \in {\cal A}} 
\mathbb{E} \edg{\int_0^T f(t,X^{\alpha,\beta},\alpha_t) dt + \int_0^T g(t,X^{\alpha,\beta},\alpha_t) d\beta_t 
+ h(X^{\alpha,\beta})}
\end{equation}
for a constant time horizon $T \in \mathbb{R}_+$ and a $d$-dimensional controlled process evolving according to
\begin{equation} \label{dyn}
dX^{\alpha,\beta}_t = \mu(t,X^{\alpha,\beta}, \alpha_t)dt + \nu(t,X^{\alpha,\beta}, \alpha_t) d\beta_t 
+ \sigma(t,X^{\alpha,\beta}) dW_t, \quad X_0 = x \in \mathbb{R}^d,
\end{equation}
where $(W_t)$ is an $n$-dimensional Brownian motion on a probability space 
$(\Omega, {\cal F}, \p)$ with corresponding augmented filtration $\mathbb{F} = ({\cal F}_t)$.
The set of controls ${\cal A}$ consists of pairs $(\alpha, \beta)$, where $(\alpha_t)_{0 \le t \le T}$ is an 
$\mathbb{F}$-predictable process with values in a compact subset $A \subseteq \mathbb{R}^k$ (the regular control) 
and an $l$-dimensional $\mathbb{F}$-adapted continuous process $(\beta_t)$ with nondecreasing components
such that $\beta_0 = 0$ and $\beta_T \in L^2(\p)$ (the singular control). The coefficients $\mu, \nu, \sigma$ 
and the performance functions $f,g,h$ can depend in a non-anticipative way on the paths of $X^{\alpha,\beta}$.
Depending on their exact specification, there might exist an optimal control in ${\cal A}$, or an
optimal control might require $(\beta_t)$ to jump and can only be approximated with controls in ${\cal A}$.

Let us denote by $C^d$ the space of all continuous functions from $[0,T]$ to $\mathbb{R}^d$ and set
$$
\|x\|_t := \sup_{0 \le s \le t} |x_s|, \quad x \in C^d,
$$
where $|.|$ is the Euclidean norm on $\mathbb{R}^d$. We make the following

\begin{Assumption} \label{ass} $\mbox{}$\\[-4mm]
\begin{enumerate}
\item[{\rm (i)}] $\sigma \colon [0,T] \times C^d \to \mathbb{R}^{d \times n}$ is a measurable function such that 
$$
\int_0^T |\sigma(t,0)|^2 dt < \infty \quad \mbox{and} \quad 
|\sigma(t,x) - \sigma(t,y)| \le L \|x-y\|_t \mbox{ for some constant } L \in \mathbb{R}_+.
$$
\item[{\rm (ii)}]
$\mu$ is of the form $\mu(t,x,a) = \eta(t,x) + \sigma(t,x) \tilde{\mu}(t,x,a)$ for measurable functions 
$\eta \colon [0,T] \times C^d \to \mathbb{R}^d$ and $\tilde{\mu} \colon [0,T] \times C^d \times A \to \mathbb{R}^n$ 
such that 
$$
\int_0^T |\eta(t,0)|^2 dt < \infty \quad \mbox{and} \quad 
|\eta(t,x) - \eta(t,y)| \le L \|x-y\|_t \mbox{ for some constant } L \in \mathbb{R}_+,
$$
$\tilde{\mu}(t,x,a)$ is bounded and continuous in $a$ and 
$$
\int_0^T \sup_{a \in A} |\mu(t,0,\alpha_t)|^2 dt < \infty, \quad
\sup_{a \in A} |\mu(t,x,a) - \mu(t,y,a)| \le L \|x-y\|_t \mbox{ for some } L \in \mathbb{R}_+.
$$
\item[{\rm (iii)}]
$\nu$ is of the form $\nu(t,x,a) = \sigma(t,x) \tilde{\nu}(t,x,a)$ for a measurable function
$\tilde{\nu} \colon [0,T] \times C^d \times A \to \mathbb{R}^{d \times n}$ such that 
$\tilde{\nu}(t,x,a)$ is bounded and continuous in $a$ and 
$$
\int_0^T \sup_{a \in A} |\nu(t,0,\alpha_t)|^2 dt < \infty, \quad
\sup_{a \in A} |\nu(t,x,a) - \nu(t,y,a)| \le L \|x-y\|_t \mbox{ for some } L \in \mathbb{R}_+.
$$
\item[{\rm (iv)}] 
The functions $f,g \colon [0,T] \times C^d \times A \to \mathbb{R}$ are measurable;
$f(t,x,a)$ and $g(t,x,a)$ are non-anticipative in $x$ and upper semicontinuous in $(x,a)$;
$h \colon C^d \to \mathbb{R}$ is upper semicontinuous in $x$; and the supremum in \eqref{problem} 
is finite.
\end{enumerate} 
\end{Assumption}

Under these assumptions, equation \eqref{dyn} has for every pair $(\alpha, \beta) \in {\cal A}$
a unique strong solution $(X^{\alpha,\beta}_t)$, and the SDE 
\be \label{Xsde}
dX_t = \eta(t,X) dt + \sigma(t,X) dW_t
\ee
has a unique strong solution $(X_t)$; see e.g. Protter (2004).

For our main representation result, Theorem \ref{thm:sbsde}, 
we need the mappings $p \colon [0,T] \times C^d \times \mathbb{R}^n \to \mathbb{R}$ 
and $q \colon [0,T] \times C^d \times \mathbb{R}^n \to \mathbb{R}^l$ given by 
$$
p(t,x,z) := \sup_{a \in A} \crl{f(t,x,a) + z \tilde{\mu}(t,x,a)} \quad \mbox{and} \quad
q_i(t,x,z) := \sup_{a \in A} \crl{g_i(t,x,a) + \sum_{j=1}^n z_j \tilde{\nu}_{ji}(t,x,a)}.
$$
In this paper, a supersolution of the BSDE
$$
Y_t = h(X) + \int_t^T p(s,X,Z_s) ds - \int_t^T Z_s dW_s 
\quad \mbox{with constraint} \quad q(t,X,Z_t) \in \mathbb{R}^l_-
$$
consists of a triplet $(Y,Z,K) \in {\cal S}^2 \times {\cal H}^2 \times {\cal K}^2$ such that
\begin{equation} \label{cBSDE}
Y_t = h(X) + \int_t^T p(s,X,Z_s) ds +(K_T-K_t) - \int_t^T Z_s dW_s  \quad \mbox{and}
\quad q(t,X,Z_t) \in \mathbb{R}^l_-  \quad \mbox{for all } t, 
\end{equation}
where 
\begin{itemize}
\item
${\cal S}^2$ is the space of $d$-dimensional RCLL $\mathbb{F}$-adapted processes $(Y_t)$ such that 
$\mathbb{E} \sup_{0 \le t \le T} |Y_t|^2 < \infty$,
\item
${\cal H}^2$ the space of $\mathbb{R}^{d \times n}$-valued $\mathbb{F}$-predictable processes $(Z_t)$ such that 
$\mathbb{E} \int_0^T |Z_t|^2 dt < \infty$, and
\item
${\cal K}^2$ the set of processes $(K_t)$ in ${\cal S}^2$ with nondecreasing components starting at $0$.
\end{itemize}
Moreover, we call $(Y,Z,K)$ a minimal supersolution of \eqref{cBSDE} if
$Y_t \le Y'_t$, $0 \le t \le T$, for any other supersolution $(Y',Z',K')$; see e.g. Peng (1999).

Our main result is the following:

\begin{theorem} \label{thm:sbsde}
The constrained BSDE \eqref{cBSDE} has a minimal supersolution $(Y,Z,K)$, and $Y_0 = I$.
\end{theorem}

The next result shows that problem \eqref{I} can be approximated by restricting the controls to 
regular piecewise constant controls: for $j \in \mathbb{N}$, denote by 
${\cal A}^j$ the set of all pairs $(\alpha, \beta) \in {\cal A}$ of the form 
$\alpha = \sum_{i=0}^{m-1} a_i 1_{(t_i,t_{i+1}]}$ and $\beta_t = \int_0^t b_s ds$, where 
$b = \sum_{i=0}^{m-1} b_i 1_{(t_i,t_{i+1}]}$, the $b_i$ are ${\cal F}_{t_i}$-measurable with values in $[0,j]^l$
and $0=t_0 < t_1< t_2 \dots < t_m = T$ is a deterministic partition of $[0,T]$. The 
corresponding control problem is 
\begin{equation} \label{jproblem}
I^j := \sup_{(\alpha, \beta) \in {\cal A}^j} 
\mathbb{E} \edg{\int_0^T f(t,X^{\alpha,\beta},\alpha_t) dt + \int_0^T g(t,X^{\alpha,\beta},\alpha_t) d\beta_t 
+ h(X^{\alpha,\beta})},
\end{equation}
and the following holds:

\begin{proposition} \label{prop:appr}
One has $I^j \uparrow I$ for $j \to \infty$.
\end{proposition}

Moreover, since \eqref{jproblem} is a regular control problem, it admits a representation through a standard BSDE
\be \label{BSDE}
Y_t = h(X) + \int_t^T p^j(s,X,Z_s) ds - \int_t^T Z_s dW_s 
\ee
with a driver given by
$$
p^j(t,x,z) := \sup_{a \in A, b \in [0,j]^m} \crl{f(t,x,a) + z \tilde{\mu}(t,x,a) + [g(t,x,a) + z \tilde{\nu}(t,x,a)] b}.
$$
Compared to the constrained BSDE \eqref{cBSDE}, which gives the optimal value of the control problem \eqref{I},
the BSDE \eqref{BSDE} provides a characterization of the optimal value of \eqref{jproblem} as well as 
corresponding optimal controls.

\begin{theorem} \label{thm:appr}
For every $j \in \mathbb{N}$, BSDE \eqref{BSDE} has 
a unique solution $(Y^j,Z^j)$ in ${\cal S}^2 \times {\cal H}^2$. Moreover, $Y^j_0 = I^j$, and
for any pair of progressively measurable functionals $\hat{\alpha} \colon [0,T] \times C \to A$,
$\hat{b} \colon [0,T] \times C \to [0,j]^l$ satisfying 
$$
f(t,X,\hat{\alpha}_t(X)) + Z^j_t \tilde{\mu}(t,X,\hat{\alpha}_t(X)) + [g(t,X,\hat{\alpha}_t(X)) 
+ Z^j_t \tilde{\nu}(t,X,\hat{\alpha}_t)] \hat{b}_t(X)\\
= p^j(t,X,Z^j_t) \mbox{ $dt \times d\p$-a.e.},
$$
$\alpha_t = \hat{\alpha}_t(X^{\alpha,\beta})$ and $\beta_t = 
\int_0^t \hat{b}_s(X^{\alpha, \beta})ds$ defines a pair in ${\cal A}$ such that
$$
I^j = \mathbb{E} \edg{\int_0^T f(t,X^{\alpha,\beta},\alpha_t) dt + \int_0^T g(t,X^{\alpha,\beta},\alpha_t) d\beta_t 
+ h(X^{\alpha,\beta})}.
$$
\end{theorem}

Our last result concerns the continuity of the minimal supersolution of \eqref{cBSDE} at the final time $T$.
Due to the constraint $q$, $Y$ might jump downwards at $T$. This can be avoided by modifying $h$.
Define the face-lift $\hat{h} : C \to \mathbb{R}$ as follows 
$$
\hat{h}(x) := \inf\{h(x+\nu(T,x)l{\bf 1}_{\{T\}})+g(T,x)l,\quad l\in \mathbb R_+^d\}.
$$
Then the following holds follwing an argument from \cite{Bou14}.

\begin{proposition} \label{prop:facelift}
The BSDE 
\be \label{flBSDE}
Y_t = \hat{h}(X) + \int_t^T p(s,X,Z_s) ds - \int_t^T Z_s dW_s \quad 
\mbox{with constraint} \quad q(t,X,Z_t) \in \mathbb{R}^l_-
\ee
admits a minimal supersolution $(\hat{Y}, \hat{Z}, \hat{K})$, and one has 
$\Delta \hat{Y}_T = 0$ as well as $(\hat{Y}_t,\hat{Z}_t,\hat{K}_t) = (Y_t,Z_t,K_t)$ for $t \in [0,T)$, where 
$(Y,Z,K)$ is the minimal supersolution of \eqref{cBSDE}.
\end{proposition}



\section{Proofs}
\label{sec:proofs}

We start with the\\[2mm]
{\bf Proof of Proposition \ref{prop:appr}}. It is straightforward to see that $I^j$ is nondecreasing and
$I^j\le I$. By a density argument, we can prove that 
$$\lim_j I^j=I.$$
\qed

Next, we show that the approximate problems \eqref{jproblem} admit a weak formulation.
To do that, we note that by Girsanov's theorem, the process 
$$
W^{\alpha, \beta}_t := W_t - \int_0^t \edg{\tilde{\mu}(s,X,\alpha_s) + \tilde{\nu}(s,X,\alpha_s) b_s} ds.
$$
is for every pair $(\alpha, \beta) \in {\cal A}^j$, a Brownian motion under the measure $\p^{\alpha, \beta}$ given by 
$$
\frac{d\p^{\alpha, \beta}}{d\p} = {\cal E} \brak{\int_0^. \edg{\tilde{\mu}(s,X,\alpha_s) + \tilde{\nu}(s,X,\alpha_s) b_s} dW_s}.
$$
Moreover, the following holds:

\begin{lemma} \label{lemma:filtration}
For all $(\alpha, \beta) \in {\cal A}^j$, the augmented filtration generated by $W^{\alpha, \beta}$ equals $\mathbb{F}$.
\end{lemma}

\begin{proof}
Denote the augmented filtration of $W^{\alpha,\beta}$ by
$\mathbb{F}^{\alpha,\beta} = ({\cal F}^{\alpha,\beta}_t)$.
Since $X$ is a strong solution of the SDE \eqref{Xsde}, it is $\mathbb{F}$-adapted. So it follows from the 
definition of $W^{\alpha,\beta}$ that $\mathbb{F}^{\alpha,\beta}$ is contained in $\mathbb{F}$.

On the other hand, one has $\alpha = \sum_{i=0}^{m-1} a_i 1_{(t_i,t_{i+1}]}$ and 
$b = \sum_{i=0}^{m-1} b_i 1_{(t_i,t_{i+1}]}$ for $a_i$ and $b_i$ ${\cal F}_{t_i}$-measurable. In particular $a_0$ and $b_0$ are deterministic. So it follows from Assumption \ref{ass} that on $[0,t_1]$, $(X_t)$ is the unique strong solution of 
$$
dX_t = \mu(t,X,\alpha_t) dt + \nu(t,X,\alpha_t)b_t dt + \sigma(t,X) dW^{\alpha,\beta}_t, \quad X_0 = x.
$$
Hence, $(X_t)_{t \in [t_0,t_1]}$ is $({\cal F}^{\alpha, \beta}_t)_{t \in [t_0,t_1]}$-adapted, from which it follows that
$$
W_t = W^{\alpha, \beta}_t + \int_0^t \edg{\tilde{\mu}(s,X,\alpha_s) + \tilde{\nu}(s,X,\alpha_s) b_s} ds, 
\quad t \in [0,t_1],
$$
is $({\cal F}^{\alpha,\beta}_t)_{t \in [0,t_1]}$-adapted. This shows that $a_1$ and $b_1$ are 
${\cal F}^{\alpha, \beta}_{t_1}$-measurable. Now the lemma follows by induction over $i$.
\end{proof}

Using Lemma \ref{lemma:filtration}, one can derive the following weak formulation of problem \eqref{jproblem}:

\begin{lemma} One has
\begin{equation} \label{weak}
I^j = \sup_{(\alpha,\beta) \in {\cal A}^j} \mathbb{E}^{\alpha,\beta}
\edg{\int_0^T f(t,X,\alpha_t) dt + \int_0^T g(t,X,\alpha_t) d\beta_t + h(X)},
\end{equation}
where $\mathbb{E}^{\alpha,\beta}$ denotes the expectation under $\p^{\alpha,\beta}$.
\end{lemma}

\begin{proof}
For all $(\alpha, \beta) \in {\cal A}^j$, $X^{\alpha, \beta}$ is the unique strong solution of
$$
dX^{\alpha, \beta}_t = \mu(t,X^{\alpha,\beta},\alpha_t) dt + \nu(s,X^{\alpha, \beta},\alpha_t) b_t dt
+ \sigma(t,X^{\alpha, \beta}) dW_t, \quad X_0 = x,
$$
and $X$ the unique strong solution of 
$$
dX_t = \mu(t,X,\alpha_t) dt + \nu(s,X,\alpha_t) b_t dt + \sigma(t,X) dW^{\alpha, \beta}_t, \quad X_0 = x.
$$
Since $a_0$ and $b_0$ are deterministic, $(\alpha_t, \beta_t ,X_t)_{t \in [0,t_1]}$ has the same distribution 
under the measure $\p^{\alpha, \beta}$ as $(\alpha_t, \beta_t, X^{\alpha, \beta}_t)_{t \in [0,t_1]}$ under $\p$.
Moreover, $a_1$ and $b_1$ are functions of $(W_t)_{t \in [0,t_1]}$. So if one defines $\tilde{a}_1$ and $\tilde{b}_1$ 
to be the same functions of $(W^{\alpha,\beta}_t)_{t \in [0,t_1]}$, then $(\tilde{\alpha}_t, \tilde{\beta}_t, X_t)_{t \in [t_1,t_2]}$ 
has the same distribution under $\p^{\tilde{\alpha}, \tilde{\beta}}$ as
$(\alpha_t, \beta_t, X^{\alpha, \beta}_t)_{t \in [t_1,t_2]}$ under $\p$. Continuing like this, 
one sees that for every pair $(\alpha, \beta) \in {\cal A}^j$, there exists a pair 
$(\tilde{\alpha}, \tilde{\beta}) \in {\cal A}^j$ such that $(\tilde{\alpha}, \tilde{\beta}, X)$ 
has the same distribution under $\p^{\tilde{\alpha}, \tilde{\beta}}$ as $(\alpha, \beta, X^{\alpha, \beta})$ under $\p$.
Conversely, it can be deduced from Lemma \ref{lemma:filtration} with the same argument that 
for every pair $(\alpha, \beta) \in {\cal A}^j$, there exists a pair $(\tilde{\alpha}, \tilde{\beta}) \in {\cal A}^j$ 
such that $(\tilde{\alpha}, \tilde{\beta}, X^{\tilde{\alpha}, \tilde{\beta}})$ has the same distribution under $\p$ 
as $(\alpha, \beta, X)$ under $\p^{\alpha, \beta}$. This proves the lemma.
\end{proof}

$\mbox{}$\\
Now, we are ready to give the\\[2mm]
{\bf Proof of Theorem \ref{thm:appr}}\\
It follows from our assumptions that the BSDE \eqref{BSDE} satisfies the standard conditions.
So it has a unique solution $(Y^j,Z^j)$ in ${\cal S}^2 \times {\cal H}^2$; see e.g. ...
For each pair $(\alpha, \beta) \in {\cal A}^j$, we set
$$
Y^{\alpha,\beta}_t := \mathbb{E}^{\alpha,\beta}
\edg{\int_t^T f(s,X,\alpha_s) ds + \int_t^T g(t,X,\alpha_s) d\beta_s + h(X) \mid {\cal F}_t}.
$$
By Lemma \ref{lemma:filtration} and the predictable representation theorem 
that there exists an $\mathbb{R}^n$-valued $\mathbb{F}$-predictable process $Z^{\alpha,\beta}$ such that 
$$
\mathbb{E}^{\alpha,\beta}
\edg{\int_0^T f(s,X,\alpha_s) ds + \int_0^T g(t,X,\alpha_s) d\beta_s + h(X) \mid {\cal F}_t}
= Y^{\alpha,\beta}_0 + \int_0^t Z^{\alpha,\beta}_s dW^{\alpha,\beta}_s.
$$
Hence,
\beas
Y^{\alpha, \beta}_t 
&=& h(X) + \int_t^T f(s,X,\alpha_s) ds + \int_t^T g(s,X,\alpha_s) d\beta_s -
\int_t^T Z^{\alpha,\beta}_s dW^{\alpha,\beta}_s\\
&=& h(X) + \int_t^T \crl{f(s,X,\alpha_s) + Z^{\alpha,\beta}_s \tilde{\mu}(s,X,\alpha_s)} ds\\
&& + \int_t^T \crl{g(s,X,\alpha_s) + Z^{\alpha,\beta}_s \tilde{\nu}(s,X,\alpha_s)} d\beta_s -
\int_t^T Z^{\alpha,\beta}_s dW_s.
\eeas
By a comparison result for BSDEs (see e.g. ...), one has $Y^j \ge Y^{\alpha,\beta}$. On the other hand,
it can be deduced from a measurable selection argument that there exist progressively measurable functions 
$$
\tilde{\alpha} : [0,T] \times C \times \mathbb{R}^n \to A \quad \mbox{and} \quad 
\tilde{b} : [0,T] \times C \times \mathbb{R}^n \to [0,j]^l
$$
such that 
$$
f(t,x,\tilde{\alpha}(t,x,z)) + z \tilde{\mu}(t,x,\tilde{\alpha}(t,x,z)) 
 + [g(t,x,\tilde{\alpha}(t,x,z)) + z \tilde{\nu}(t,x,\tilde{\alpha}(t,x,z))] \tilde{b}(t,x,z) = p^j(t,x,z).
$$
$\alpha_t = \tilde{\alpha}(t,X,Z^j)$ and $\beta_t = \int_0^t \tilde{b}(s,X,Z^j_s) ds$ defines a pair in ${\cal A}$ 
that can be approximated by a sequence of pairs $(\alpha^n,\beta^n) \in {\cal A}^j$ in $L^2$. Then 
$$
\mathbb{E}^{\alpha^n,\beta^n}
\edg{\int_0^T f(s,X,\alpha^n_s) ds + \int_0^T g(t,X,\alpha^n_s) d\beta^n_s + h(X)}
$$
converges to
$$
\mathbb{E}^{\alpha,\beta} 
\edg{\int_0^T f(s,X,\alpha_s) ds + \int_0^T g(t,X,\alpha_s) d\beta_s + h(X)},
$$
and
\beas
Y^j_0 &=& h(X) + \int_0^T p^j(s,X,Z^j_s) ds - \int_0^T Z^j_s dW_s\\
&=& h(X) + \int_0^T \crl{f(s,X,\alpha_s) + Z^j_s \tilde{\mu}(s,X,\alpha_s)} ds\\
&& + \int_0^T \crl{g(s,X,\alpha_s) + Z^j_s \tilde{\nu}(s,X,\alpha_s)} b_s ds - \int_0^T Z^j_s dW_s\\
&=& h(X) + \int_0^T \crl{f(s,X,\alpha_s) + g(s,X,\alpha_s) b_s} ds - \int_0^T Z^j_s dW^{\alpha,\beta}_s \\
&=& \mathbb{E}^{\alpha,\beta} \edg{\int_0^T f(s,X,\alpha_s) ds + \int_0^T g(t,X,\alpha_s) d\beta_s + h(X)}.
\eeas
This shows that $Y^j_0 = I^j$.

Finally, if $\hat{\alpha} \colon [0,T] \times C \to A$ and $\hat{b} \colon [0,T] \times C \to [0,j]^l$ are
progressively measurable functionals such that
$$
f(t,X,\hat{\alpha}_t(X)) + Z^j_t \tilde{\mu}(t,X,\hat{\alpha}_t(X)) + [g(t,X,\hat{\alpha}_t(X)) 
+ Z^j_t \tilde{\nu}(t,X,\hat{\alpha}_t)] \hat{b}_t(X)\\
= p^j(t,X,Z^j_t) \mbox{ $dt \times d\p$-a.e.},
$$
it follows as above that 
$$
Y^j_0 = \mathbb{E}^{\hat{\alpha}(X),\hat{\beta}(X)} \edg{\int_0^T f(s,X,\hat{\alpha}(X)_s) ds 
+ \int_0^T g(t,X,\hat{\alpha}(X)_s) d\hat{\beta}(X)_s + h(X)}.
$$
Moreover, $\alpha_t = \hat{\alpha}_t(X^{\alpha,\beta})$ and $\beta_t = 
\int_0^t \hat{b}_s(X^{\alpha, \beta})ds$ defines a pair in ${\cal A}$ such that
$(\alpha,\beta,X^{\alpha,\beta})$ has the same distribution under $\p$ as 
$(\hat{\alpha}(X), \hat{\beta}(X),X)$ under $\p^{\hat{\alpha}(X),\hat{\beta}(X)}$.
As a consequence,
$$
I^j = \mathbb{E} \edg{\int_0^T f(t,X^{\alpha,\beta},\alpha_t) dt + \int_0^T g(t,X^{\alpha,\beta},\alpha_t) d\beta_t 
+ h(X^{\alpha,\beta})},
$$
and the proof is complete.
\qed
$\mbox{}$\\
{\bf Proof of Theorem \ref{thm:sbsde}}\\
We know that $I^j \uparrow I$ and $I^j = Y^j_0$, where $(Y^j,Z^j)$ is the solution of \eqref{BSDE}.
On the other hand, it follows from Peng (1999) that $Y^j$ increases to $Y$, where $(Y,Z)$ is the maximal 
subsolution of \eqref{cBSDE}. \qed

\end{document}